\documentclass[12pt, reqno]{amsart}
\usepackage {amsmath,amssymb,verbatim}
\usepackage[pagebackref,pdftex,hyperindex]{hyperref}

\usepackage{lineno}

\newcommand*\patchAmsMathEnvironmentForLineno[1]{%
  \expandafter\let\csname old#1\expandafter\endcsname\csname #1\endcsname
  \expandafter\let\csname oldend#1\expandafter\endcsname\csname end#1\endcsname
  \renewenvironment{#1}%
     {\linenomath\csname old#1\endcsname}%
     {\csname oldend#1\endcsname\endlinenomath}}% 
\newcommand*\patchBothAmsMathEnvironmentsForLineno[1]{%
  \patchAmsMathEnvironmentForLineno{#1}%
  \patchAmsMathEnvironmentForLineno{#1*}}%
\AtBeginDocument{%
\patchBothAmsMathEnvironmentsForLineno{equation}%
\patchBothAmsMathEnvironmentsForLineno{align}%
\patchBothAmsMathEnvironmentsForLineno{flalign}%
\patchBothAmsMathEnvironmentsForLineno{alignat}%
\patchBothAmsMathEnvironmentsForLineno{gather}%
\patchBothAmsMathEnvironmentsForLineno{multline}%
}

\usepackage[margin=3cm]{geometry}
\newcommand{\A}{\mathbb{A}}

\newcommand{\C}{\mathbb{C}}

\newcommand{\N}{\mathbb{N}}

\renewcommand{\P}{\mathbb{P}}

\newcommand{\R}{\mathbb{R}}
\renewcommand{\S}{\mathbb{S}}

\newcommand{\Z}{\mathbb{Z}}

\newcommand{\cE}{\mathcal{E}}

\newcommand{\cH}{\mathcal{H}}

\newcommand{\cL}{\mathcal{L}}

\newcommand{\cW}{\mathcal{W}}
\newcommand{\cX}{\mathcal{X}}

\renewcommand{\b}{\beta}

\renewcommand{\d}{\delta}
\newcommand{\e}{\varepsilon}

\renewcommand{\phi}{\varphi}

\renewcommand{\leq}{\leqslant}
\renewcommand{\geq}{\geqslant}
\newcommand{\abs}[1]{\left\lvert#1\right\rvert}
\newcommand{\norm}[1]{\left\|#1\right\|} 

\newcommand{\bal}{\mathrm{b}}

\renewcommand{\DH}{\mathrm{DH}}

\newcommand{\KE}{\mathrm{KE}}
\newcommand{\NA}{\mathrm{NA}}
\renewcommand{\Re}{\mathrm{Re}}

\DeclareMathOperator{\Aut}{Aut}

\DeclareMathOperator{\dd}{\sqrt{-1}\partial\bar{\partial}}

\DeclareMathOperator{\Ent}{Ent}

\DeclareMathOperator{\GL}{GL}

\DeclareMathOperator{\Ric}{Ric}

\DeclareMathOperator{\Tr}{Tr}

\numberwithin{equation}{section}       % Number formulas within sections

\newtheorem{prop} {Proposition} [section]
\newtheorem{thm}[prop] {Theorem} 
\newtheorem{dfn}[prop] {Definition}
\newtheorem{lem}[prop] {Lemma}
\newtheorem{cor}[prop]{Corollary}

\newtheorem{rem}[prop]{Remark}
\theoremstyle{remark}

\newtheorem*{thmA}{\bf{Theorem A}} 
\newtheorem*{thmB}{\bf{Theorem B}} 
\newtheorem*{thmC}{\bf{Theorem C}}

\newtheorem*{dfn*}{\bf{Definition}}

\title[]{Quantization of the K\"ahler-Ricci flow and optimal destabilizer for a Fano manifold} 
\date{\today} 

\author{Tomoyuki Hisamoto}
\address{Tokyo Metropolitan University\\
Minami-Ohsawa\\
Tokyo\\ 
Japan}
\email{hisamoto@tmu.ac.jp}

\begin{document}

\maketitle

\maketitle
%\linenumbers
\setcounter{tocdepth}{1}

\begin{abstract}
For a Fano manifold, 
We consider the geometric quantization of the K\"ahler-Ricci flow 
and the associated entropy functional. 
Convergence to the original flow and entropy is established. 
It is also possible to 
formulate the finite-dimensional analogue of the optimal degeneration for the anti-canonical polarization. 
\end{abstract}

\tableofcontents

%%%%%%%%%%%%%%%%%%%%%%%%%%%%%%%%%%%%%%%%%%%%%%%%%%%%%%%%%%%%%%%%%%%%%%%%%%%%%%%%%%
\section{Introduction}
\label{Introduction}
%%%%%%%%%%%%%%%%%%%%%%%%%%%%%%%%%%%%%%%%%%%%%%%%%%%%%%%%%%%%%%%%%%%%%%%%%%%%%%%%%%

The Ricci flow introduced by Hamilton, 
used in Perelman's solution of the Poincar\'e conjecture 
and Brendle-Schoen's proof of the differentiable sphere theorem, 
has become powerful tool in modern differential geometry. 
In K\"ahler geometry, we should consider the K\"ahler-Ricci flow, 
which has recently atracted people's attention due to 
its remarkable applications in the problems of K\"ahler-Einstein metrics. 
In this paper, we formulate the geometric quantization 
of the K\"ahler-Ricci flow and discuss its basic properties and some applications.
We are particularly motivated by the optimal degeneration problem 
(explained in subsection \ref{section: optimal degeneration}), 
the multiplier ideal sheaf construction of \cite{His23}, 
and the study of delta-invariant by \cite{RTZ20}. 
For this reason the expositions are especially devoted to the Fano manifold setting. 
One can obtain the similar formulations for the Calabi-Yau
 or the canonically polarized manifolds. 
At any rate let $X$ be a complex Fano manifold 
so that the anti-canonical line bundle $-K_X$ is ample. 
Here we adopt the additive notation in taking the tensor 
of line bundles and denote the dual line bundle by the minus sign. 
The idea of the geometric quantization tells us that 
K\"ahler metrics are quantized to Hermitian forms 
on the finite-dimensional vector space of global sections $H^0(X, -k K_X)$ such that 
as the ``Planck constant" $1/k$ approaches to zero 
the quantized objects approximate the classical ones. 
Let $\cH$ be the collection of K\"ahler metrics with prescribed cohomology class 
$c_1(X)$. 
We denote the collection of the above Hermitian forms by $\cH_k$. 
The remarkable point is that one indeed has the taking Fubini-Study type metric map 
\begin{equation}
    f_k \colon \cH_k \to \cH, 
\end{equation}  
which directly transfers the quantized objects to the classical ones. 
Further, taking $L^2$ norms of sections one obtains the projection
$p_k \colon \cH \to \cH_k$ in the opposite direction. 
The composition map $b_k := p_k \circ f_k$ which we call the balancing map 
measures how the qunatized object is far from the quantized Einstein condition. 
Since $b_k(H)$ is positive Hermitian for any $H \in \cH_k$, 
one can takes the logarithm $\log{b_k(H)}$. 
as we will see in section \ref{Quantization}, 
this logarithm matrix can be seen as the quantization of the Ricci potential 
$\rho=\rho_{\omega}$ of the given K\"ahler metric $\omega$. 
Now we introduce the quantization of the (normalized) K\"ahler-flow 
as the evolution of the Hermitian forms $H=H_t$ which satisfies 

\begin{equation}\label{quantized flow 0}
\frac{1}{k}\frac{d}{dt}\log{H} = -\log{b_k(H)}+\log{H}. 
\end{equation}

Our first observation is that the quantized flow 
indeed approximates the original K\"ahler-Ricci flow as $k \to \infty$. 

\begin{thmA}\label{Theorem A}
    Let $\phi_t \in \cH$ be the K\"ahler-Ricci flow (\ref{KRF}) and $H_t$ be the quantized flow (\ref{quantized flow 0}) initiated from $H_0=p_k(\phi_0)$. 
    Then for any $T>0$ there exists a constant $C_T$ such that 
    \begin{equation}
    \abs{\phi_{(j+1)/k} -f_k(H_{j/k})} \leq C_T k^{-1}
    \end{equation}
    holds for any $(j+1)/k \leq T$. 
\end{thmA}
For the proof we compare the quantized flow with the Bergman iteration of
 \cite{Berm13}. 
In fact the latter is obtained as the difference equation of the former. 
Theorem A then follows from Berman's corresponding result for the Bergman iteration. 

In the study of the K\"ahler-Ricci flow, 
the most crucial steps are the Perelman's introduction of 
the $\cW$-functional and its monotonicity.
In K\"ahler geometry \cite{He16} defined the reduced entropy functional 
\begin{equation}
    S(\omega) = \int_X \rho e^\rho \omega^n 
\end{equation} 
and its monotonicity along the K\"ahler-Ricci flow. 
As the next subject of the present article 
we introduce the quantized version of He's entropy functional, as 
\begin{equation}
    S_k(H) := \Tr b_k(H)(\log{b_k(H)}-\log{H}). 
\end{equation} 
This is in fact equivalent to the relative entropy functional $S(A \vert B)$ 
frequently used in the quantized information theory. 
The quantized entropy $S_k$ seems not to enjoy the perfect monotonicity, 
however, we still have other desired properties and the following 
convergence result. 

\begin{thmB}\label{Theorem B}
    Let $\phi$ be a K\"ahler metric and $p_k(\phi)$ be the associated $L^2$-inner product 
    defined on the space of anti-canonical sections. 
    Then we have the convergence of the entropies: 
    \begin{equation}
    \lim_{k \to \infty}S_k(p_k(\phi)) =S(\phi). 
    \end{equation}
\end{thmB}

Our third point is the relation with the optimal degeneration for a Fano manifold. 
As we will explain in subsection \ref{section: optimal degeneration}, 
the test configuration $(\cX, \cL)$ is called {\em optimal} if it minimizes 
the non-Archimedean entropy $S^\NA(\cX, \cL)$ among all other degenerations. 
According to \cite{BHJ17} and the work of Boucksom-Jonsson, 
the collection of test configurations $\cH^\NA$ indeed coincides non-Archimedean 
metrics of the Berkovich analytification of the polarization $(X, -K_X)$. 
We define the quantization $S^\NA_k(\nu)$ for each $\nu \in \cH_k^\NA$ and 
formulate the optimal non-Archimedean norm in this finite-dimensional settings.   
\begin{thmC}\label{Theorem C}
    We have the quantized version of the equality (\ref{optimal degeneration}) so that   
    \begin{equation}
    \inf_{H \in \cH_k} S_k(H) = \max_{\nu \in \cH_k^\NA} (-S_k^\NA(\nu))
    \end{equation}
    holds for any fixed $k \in \N$. 
\end{thmC} 

It would be interesting to compute the above 
optimal non-Archimedean norm for the manifolds with large symmetry, such as homogeneous or toric varieties. 

%%%%%%%%%%%%%%%%%%%%%%%%%%%%%%%%%%%%%%%%%%%%%%%%%%%%%%%%%%%%%%%%%%%%%%%%%%%%%%%%%%
\section{K\"ahler-Ricci flow and the optimal degeneration}
\label{Kahler-Ricci flow and optimal degeneration}
%%%%%%%%%%%%%%%%%%%%%%%%%%%%%%%%%%%%%%%%%%%%%%%%%%%%%%%%%%%%%%%%%%%%%%%%%%%%%%%%%%

\subsection{K\"ahler-Ricci flow}
\label{section: Kahler-Ricci flow}
Let $(X, L)$ be an complex $n$-dimensional polarized algebraic manifold. We are mainly interested in the Fano case where the polarization $L=-K_X$ is the anti-canonical line bundle. 
In this setting \textit{geometric flow} means a certain time-evolution of K\"ahler metrics $\omega=\omega(t)$  in the first Chern class $c_1(L)$. 
A typical example is the normalized K\"ahler-Ricci flow equation:  
\begin{equation}\label{KRF}
\frac{\partial}{\partial t} \omega = -\Ric{\omega} +\omega.  
\end{equation} 
From the result of \cite{Cao85}, the solution exists for all $0 \leq t <\infty$ and converges to the K\"ahler-Einstein metric if it exists. 
Using $\partial\bar{\partial}$-lemma we may rephrase the above equation in terms of the potential function.  Indeed once a reference metric $\omega_0$ is fixed one may take a function $\phi$ such that $\omega=\omega_\phi = \omega_0 +\dd\phi $. Similarly one has a function $\rho$ such that $\Ric{\omega} -\omega =\dd \rho $. We choose the normalization
\begin{equation}\label{Ricci potential}
\int_X (e^\rho-1) \omega^n =0 
\end{equation} 
so that $\rho=\rho_\omega$ is uniquely determined by $\omega$ and the equation (\ref{KRF}) is translated into 
\begin{equation}\label{KRF2}
\frac{\partial}{\partial t} \phi =- \rho. 
\end{equation} 
In the analysis of the K\"ahler-Ricci flow it is also important to rephrase the equation in terms of the associated measures. 
The first measure we have in mind is the Monge-Amp\`ere measure $\omega^n=\omega_\phi^n$ of $\phi$. 
We often divide it by the volume $V=\int_X \omega_\phi^n $ to consider the probability measure $V^{-1} \omega_\phi^n$. 
In order to encode the information of the Ricci curvature one needs another probability measure 
\begin{equation}\label{canonical measure}
d\mu_\phi = \frac{e^{-\phi} d\mu_0}{\int_X e^{-\phi}d\mu_0}, \ \ \ d\mu_0 = e^{\rho_0}\omega_0^n,  
\end{equation} 
which we call \textit{the canonical measure}. 
The equation (\ref{KRF}) now can be written down to the second order partial differential equation 
\begin{equation}\label{KRF3}
\frac{\partial}{\partial t} \phi = -\log\bigg[\frac{d\mu_\phi}{V^{-1}\omega_\phi^n}\bigg]. 
\end{equation} 

Following \cite{He16}, we introduce the entropy functional 
\begin{equation}\label{entropy}
S(\omega) = \int_X \rho e^\rho \omega^n  = \int_X \log\bigg[\frac{d\mu_\phi}{V^{-1}\omega_\phi^n}\bigg] \mu_\phi  
\end{equation}
which is actually the relative entropy $\Ent(\mu \vert \nu) = \int \log\big[ \frac{d\mu}{d\nu}\big] d\mu$ defined for two probability measures $\mu, \nu$. 
By the definition $S(\omega)$ is non-negative. One can also check that the entropy is non-increasing along the K\"ahler Ricci flow (\cite{Pali08, PSSW09}) and that $S(\omega)=0$ if and only if $\omega$ is K\"ahler-Einstein. 
From the standard probability theory we may express the entropy as the convex conjugate form 
\begin{equation}\label{convex conjugate form0}
S(\omega) =\sup_f \bigg[ \int_X f d\mu_\phi - \log \frac{1}{V}\int_X e^f \omega_\phi^n \bigg], 
\end{equation}
where $f$ runs through arbitrary real-valued continuous functions. 
The above expression of the entropy is closely related to 
\textit{the large deviation principle} which tells us that 
the rate function measuring the rarity of atypical phenomenon is 
given in the convex conjugate form of the moment generating function. 
For the exposition and further development of this theme, 
we refer Berman's work \cite{Berm13b}, \cite{Berm18}, \cite{Berm20}.

\subsection{Space of K\"ahler metrics}
\label{section: space of Kahler metrics}  
Once a reference metric $\omega_0 \in c_1(L)$ is fixed the space of K\"ahler metric is identified with the collection of potentials 
\begin{equation}
\cH = \cH(X, \omega_0)
= \bigg\{ \phi \in C^\infty(X; \R): \omega_0+\dd \phi >0 \bigg\}. 
\end{equation} 
If one take a fiber metric $h_0$ of $L$ such that 
the Chern curvature is $\omega_0$, any other fiber metric can be 
written in a form $h=h_0e^{-\phi}$. 
If there is no fear of confusing, 
it is convenient to simply denote such a fiber metric by $e^{-\phi}$. 
Thus the space $\cH$ is also identified with the collection of 
all fiber metrics with positive curvature.

The tangent space at $\phi \in \cH$ is naturally identified with $C^\infty(X; \R)$ and is equipped with the canonical but non-trivial Riemannian metric 

\begin{equation}
u \mapsto \frac{1}{V}\int_X u^2 \omega_\phi^n. 
\end{equation} 
Indeed the geodesic curvature of a given curve $\phi_t $ $(a \leq t \leq b)$ is computed as 
\begin{equation}
c(\phi) = \ddot{\phi} - \abs{\bar{\partial} \dot{\phi} }^2. 
\end{equation} 
Let us take the annulus $A = \{ \tau \in \C: e^{-b} \leq \abs{\tau} \leq e^{-a}\}$. 
The remarkable idea of \cite{Sem92} is that the curve $\phi_t$ should be identified with a $\S^1$-invariant function $\Phi$ on the product space $A \times X$ in the manner 
\begin{equation}
\Phi(\tau, z) = \phi_{-\log\abs{\tau}}(z).   
\end{equation} 
In terms of $\Phi$, geodesic equation for $\phi_t$ is transerated into the degenerate Monge-Amp\`ere equation 
\begin{equation}\label{degenerate MA equation}
(p_2^*\omega_0 + \dd \Phi)^{n+1} \equiv 0,  
\end{equation} 
where the differential operators are taken for $(n+1)$-complex variables $(\tau, z)$. 

In general the geodesic connecting given endpoints
 $\phi_a, \phi_b$ is not necessarily contained in $\cH$ and we need to consider the weak solution of (\ref{degenerate MA equation}). 
Singular metrics are also indispensable to construct the completion of the space $\cH$ where one takes a convergent subsequence of the minimizing sequence of the energy functional. 
See \cite{BBJ15} for the detail. 

The geodesity (\ref{degenerate MA equation}) 
is further equivalent to the affiness of the Monge-Amp\`ere energy 
\begin{equation}\label{MA energy}
    E(\phi)
     = \frac{1}{(n+1)V} \sum_{i=0}^n \int_X \phi \omega_0^{n-i}\wedge \omega_\phi^i, 
    \end{equation}  
which is designed to have the derivative $(d_\phi E)(u) = \frac{1}{V}\int_X u \omega_\phi^n$. 
It is also possible to define $E(\phi) \in [-\infty, \infty)$ for singular metrics. 
A metric with the property $E(\phi) >-\infty$ is called {\em finite energy} 
and enjoys many nice properties in the pluripotential theory. 
See \cite{GZ17} for the exposition. 

\subsection{Optimal degeneration}
\label{section: optimal degeneration}

The famous Yau-Tian-Donaldson conjecture states that a polarized manifold $(X, L)$ admits a constant scalar curvature K\"ahler metric in $c_1(L)$ if and only if it is K-polystable. 
To ensure the existence of such a standard metric many authors propose much more stronger stability condition but we do not enter the topic here. 
K-stability is examined by certain degenerations of $(X, L)$. 

\begin{dfn}[\cite{Don02}. See also \cite{BHJ17} for the updated terminology.]
A flat family of polarized schemes $\pi \colon (\cX, \cL) \to \P^1$ endowed with a torus action $\lambda \colon \C^* \to \Aut(\cX, \cL)$ is called \textit{test configuration} of $(X, L)$, if it satisfies the following conditions. 
\begin{itemize}
\item[$(1)$]The projection morphism $\pi$ is equivariant with respect to $\lambda$ and the natural $\C^*$-action to $\P^1$. 
\item[$(2)$]The family is equivariantly isotrivial outside $0 \in \P^1$. That is, over the affine line $\A^1$ one has $(\cX_{\A^1}, \cL_{\A^1}) \simeq (\C \times X, p_2^*L)$, where $p_2 \colon \C \times X \to X$ is the second projection. We encode the isomorphism into the datum of the test configuration. 
\item[$(3)$]The total space $\cX$ is normal. 
\end{itemize}
\end{dfn}

The original $(X, L)$ is naturally identified with the fiber $(\cX_1, \cL_1)$ over $1 \in \C$. 
For each $k \in \N$ the induced $\C^*$-action to $H^0(\cX_0, k\cL_0)$ produces the eigenvalues 
$\lambda_1, \lambda_2, \dots, \lambda_{N_k} \in \R$. Distribution of these eigenvalues in the limit $k \to \infty$ defines the Duistermaat-Heckman measure 
\begin{equation}
\DH(\cX, \cL) = \lim_{k \to \infty} \frac{1}{N_k}\sum_{i=1}^{N_k} \delta_{\frac{\lambda_i}{k}} 
\end{equation} 
associated with the test configuration. 
As $k \to \infty$, this invariant governs the leading term coefficient of the equivariant Riemann-Roch theorem on the central fiber and hence determines the leading term of the usual Riemann-Roch theorem on the total space $\cX$. 
Indeed the first moment of the Duistermaat-Heckman measure is equivalent to the self-intersection number of $\cL$: 
\begin{equation}
E^\NA(\cX, \cL) :=\frac{\cL^{n+1}}{(n+1)V} = \hat{\lambda}:= \int_{\lambda \in \R} \lambda \DH(\cX, \cL). 
\end{equation}
We call the second moment 
\begin{equation}
\norm{(\cX, \cL)}_2 = \bigg[ \int_{\lambda \in \R} (\lambda -\hat{\lambda})^2 \DH(\cX, \cL) \bigg]^\frac{1}{2}
\end{equation}
norm of the test configuration. 
If one takes a fiber metric $e^{-\Phi}$ of $\cL$, 
the information of the test configuration is encoded 
into the associated ray 
\begin{equation}
\phi_t(z) = \Phi(\lambda(e^{-t})z) 
\end{equation} 
on the space of K\"ahler metrics $\cH$. 
If such $\phi_t$ is a weak geodesic ray on $\cH$, 
we say that it is associated with the test configuration. 
For the associated weak geodesic ray, 
the $p$-th moment is transrated as 
\begin{equation}\label{DH formula}
   \int_{\lambda \in\R} \lambda^p\DH(\cX, \cL)
    = \frac{1}{V}\int_X (\dot{\phi})^p \omega^n. 
\end{equation} 
See \cite{His16, BHJ17} for the detail. 

A polarized manifold is called K-semistable if the essentially sub-leading term coefficient 

\begin{equation}\label{NA K-energy}
M^\NA(\cX, \cL) = V^{-1} K_{\cX/\P^1}^{\log} \cL^n  + \hat{S}E^\NA(\cX, \cL) 
\end{equation}
is semipositive for all $(\cX, \cL)$. 
In this article we follow the notation of \cite{BHJ17} 
in order to regard the right-hand side as the non-Archimedean version 
of the K-energy functional (and of its integral-by-parts formula \cite{Chen00b}) 

\begin{equation}\label{K-energy}
M(\phi) =  \Ent(V^{-1}\omega_\phi^n \vert V^{-1}\omega_0^n) +R(\phi) + \hat{S} E(\phi). 
\end{equation}

Here $\hat{S}$ is the average scalar curvature and $E \colon \cH \to \R$ is the Monge-Amp\`ere energy (\ref{MA energy}). 
For more precise definition of these energies and the non-Archimedean point of view we refer \cite{BHJ17, BHJ19}. At any rate we do not study the K-energy in detail here. In relation with the K\"ahler Ricci flow we prefer \textit{the non-Archimedean entropy}. 

\begin{dfn}[\cite{Berm16}, \cite{DS17}]\label{definition of S-stability}
Assume for simplicity that $\cX$ has at worst Gorenstein singularities and that the central fiber $\cX_0$ is a reduced scheme. 
We define the non-Archimedean entropy of the test configuration $(\cX, \cL)$ as 

\begin{equation}\label{NA entropy}
S^\NA(\cX, \cL) = \deg \pi_*(K_{\cX/\P^1} +\cL) + \log \int_{\lambda \in \R} e^{-\lambda} \DH(\cX, \cL).  
\end{equation}
We say $(X, L)$ is \textit{S-semistable} if $S^\NA(\cX, \cL) \geq 0 $ for all such test configurations with mild singularities. 
\end{dfn} 

Following \cite{Berm16}, one can extend the first term $L^\NA(\cX, \cL) :=\deg \pi_*(K_{\cX/\P^1} + \cL)$ to arbitrary test configurations using the log canonical threshold. 
On the other hand, the discussion in \cite{LX14} reduces the semistability to testing the sign of $S^\NA(\cX, \cL)$ for such $(\cX, \cL)$ with mild singularities. 
We see that (\ref{NA entropy}) is comparable with the definition of entropy (\ref{entropy}). 
In both definitions the first term encodes the information of the canonical divisor 
and the second term is related to the moment generating function. 
This is justified by the main result of \cite{Berm16}, 
which states that the non-Archimedean L-functional gives the 
slope of the (Archimedean) L-functional 
\begin{equation}
L(\phi) := -\log{\frac{1}{V}\int_X e^{-\phi} d\mu_0} 
\end{equation} 
along any weak geodesic ray on $\cH$, 
associated with the test configuration. 
Combined with \ref{DH formula} it implies 
\begin{equation}\label{slope formula}
    S^\NA(\cX, \cL)
    = 
    \lim_{t \to \infty} \bigg(
    \frac{d}{dt}L(\phi_t) 
    +\log{\frac{1}{V}\int_X e^{-\dot{\phi}}\omega_{\phi_t}^n}
    \bigg). 
    \end{equation}
On the other hand, a direct computation shows that 
along the K\"ahler-Ricci flow $\phi_t$ one has 
\begin{equation}
S(\phi_t) 
 = -\frac{d}{dt}L(\phi_t). 
\end{equation} 
If we keep the normalization (\ref{Ricci potential}) in mind, it indicates that the tangent $\dot{\phi}$ of the K\"ahler-Ricci flow 
attains the supremum in (\ref{convex conjugate form0}).  
We can now state what is the optimal degeneration of a Fano manifold. 
The proof of the following result employs  
either Cheeger-Colding convergence theory of Riemannian manifolds 
or Birkar's boundedness of complements. 

\begin{thm}[\cite{CSW15}, \cite{DS17}, \cite{HL20}, and \cite{BLXZ21}.]
For an arbitrary Fano manifold one has 
\begin{equation}\label{optimal degeneration}
\inf_{\omega} S(\omega)  = \max_{(\cX, \cL)} (-S^\NA(\cX, \cL)). 
\end{equation} 
There exists a unique test configuration $(\cX, \cL)$ which achieves the equality and it coincides with the first one of the 2-step degenerations generated from the Gromov-Hausdorff limit of the K\"ahler-Ricci flow $(X, \omega(t))$. 
\end{thm} 

\begin{rem}
There are various other geometric flows seeking for the standard metric. 
For example, in \cite{CHT17} the authors introduced \textit{the inverse Monge-Amp\`ere flow}

\begin{equation}
\frac{\partial}{\partial t} \phi = 1-e^\rho 
\end{equation} 

for a Fano manifold. This is the gradient flow of so-called D-energy functional 
\begin{equation}
D(\phi) = L(\phi) -E(\phi) = -\log \frac{1}{V} \int_X e^{-\phi}d\mu_0  -E(\phi) 
\end{equation}
and one has the non-Archimedean analogue  

\begin{equation}\label{NA D-energy}
D^\NA(\cX, \cL) = \deg \pi_*(K_{\cX/\P^1} + \cL) - \int_{\lambda \in \R} \lambda \DH(\cX, \cL).  
\end{equation}
The result of \cite{Berm16} guarantess that $D^\NA(\cX, \cL)$ gives the 
slope of the D-energy along the 
associated weak geodesic ray. 
Based on the result, a Fano manifold is called D-semistable 
if $D^\NA(\cX, \cL) \geq 0$ for all test configurations.
By the cerebrated work \cite{CDS15} 
together with the main result of \cite{LX14}, 
a Fano manifolds admits a K\"ahler-Einstein metric if and only if 
it is D-polystable (see also \cite{BBJ15}). 
In this case, 
the inverse Monge-Amp\`ere flow converges to the K\"ahler-Einstein metric
modulo the action of the automorphism group. 

The difference from S-stability is subtle. 
In the slope formula (\ref{slope formula}) 
one can easily see from Jensen's inequality that 
\begin{equation}
    \log \frac{1}{V}\int_X e^{-\dot{\phi}} \omega_\phi^n 
    \geq -\frac{1}{V}\int_X \dot{\phi} \omega_\phi^n 
\end{equation} 
holds hence S-stability is weaker than D-stability. 
In \cite{CHT17} we showed that for any toric Fano manifold there exists a test configuration achieving the equality 

\begin{equation}
\inf_{\omega} \bigg[ \int_X (e^\rho-1)^2 \omega^n \bigg]^\frac{1}{2} 
 = \max_{(\cX, \cL)} \frac{-D^\NA(\cX, \cL)}{\norm{(\cX, \cL)}_2},  
\end{equation}

which should be the optimal degeneration with respect to the invariant $D^\NA$. 
In the toric case the central fiber of the $D$-optimal degeneration has two irreducible component while the $S$-optimal degeneration is simply a product space. 
This is in contrast to the stable case where the K\"ahler-Einstein metric is characterized by both functionals $S$ and $D$. 
%In Table $1$ we list the other major geometric flows and corresponding invariants for test configurations. 

\end{rem} 
\section{Quantization}
\label{Quantization}
%%%%%%%%%%%%%%%%%%%%%%%%%%%%%%%%%%%%%%%%%%%%%%%%%%%%%%%%%%%%%%%%%%%%%%%%%%%%%%%%%%

\subsection{Anti-canonical settings of Geometric quantization}
\label{sec:3-1}
Let us briefly recall the framework of geometric quantization 
for the anti-canonical polarization $(X, L)$ with $L=-K_X$.

Relationship between the geometric quantization and the K\"ahler-Einstein problem 
dates back to the pioneering work of Yau, Tian, and Donaldson. 
In the Fano case one may utilize somewhat special formulation and 
we follow \cite{Berm13}, \cite{BBEGZ11} for the idea. 

For each $k \in \N $ we denote by $\cH_k$ the collection of all positive definite 
Hermitian inner products on the vector space $H^0(X, kL)$. 
From the polar decomposition for such matrices 
one obtains the homogeneous space expression $\cH_k \simeq \GL(N_k; \C)/U(N_k)$. 
The key in the geometric quantization is the map 

\begin{equation}
f_k \colon \cH_k \to \cH 
\end{equation}

which takes an orthonormal basis $s_1, s_2, \dots, s_{N_k}$ 
of $H \in \cH_k$ and then sends $H \in \cH_k$ to the metric of Fubini-Study type 
\begin{equation}
f_k(H) =\frac{1}{k} \log \frac{1}{N_k}\sum_{i=1}^{N_k} \abs{s_i}^2.  
\end{equation} 
It is easy to check that the right-hand side is independent of 
the choice of an orthonormal basis. 

One also has the converse direction map $p_k\colon \cH \to \cH_k$ taking 
of $\phi \in \cH$ the $L^2$-inner product 
\begin{equation}
p_k(\phi) = \int_X \abs{\cdot}^2 e^{-k\phi} d\mu_\phi. 
\end{equation}
The reference fiber metric $\abs{\cdot}^2=h_0(\cdot, \cdot)$ of the line bundle $L$ in the above 
is taken such that its Chern curvature equals to the fixed K\"ahler metric $\omega_0$. 
It is abuse of notation, however, 
we will simply write $h_0(s, s')$ as $s\overline{s'}$. 
See also the arrangement in subsection \ref{section: space of Kahler metrics}. 
There should be several choices of the measure, 
which yields essentially different definitions of $p_k$. 
The ``anti-canonical" setting means that we chose the canonical measure  $d\mu_\phi$ 
defined in (\ref{canonical measure}). 

We set $b_k =p_k \circ f_k$ and $\b_k =f_k \circ p_k$. In general $\b_k(\phi) \neq \phi$ but for any fixed $\phi$ the $C^\infty$-convergence 
\begin{equation}
\lim_{k \to \infty} \b_k(\phi) =\phi 
\end{equation}
holds. This is a very special case of the famous Bergman kernel asymptotic expansion. 
In this sense at least pointwisely the finite-dimensional $\cH_k$ approximates $\cH$. 
The homogeneous space $\cH_k$ is naturally endowed with the invariant Riemannian metric 
and the two spaces $\cH_k$ and $\cH$ shares analogous geometric structure 
including the expression for geodesics and the curvature tensors. 
See \cite{Guedj12} for the exposition. 
As a consequence of choosing $\mu_\phi$ as the probability measure, 
we have the following lemma. 

\begin{lem}\label{normalization lemma}
Denote by $\norm{s}_{b_k(H)}$ the norm of a section $s \in H^0(X, kL)$ 
with respect to the Hermitian form $b_k(H)$. 
For any choice of the orthonormal basis $s_i$ one has 
\begin{equation}
\sum \norm{s_i}^2_{b_k(H)} = N_k. 
\end{equation}
\end{lem} 

\begin{proof}
It is straightforward to compute  
\begin{eqnarray}
\sum \norm{s_i}^2_{b_k(H)}  &=& \sum \int_X \abs{s_i}^2 e^{-kf_k(H)} d\mu_{f_k(H)}  \\
&=& \int N_k e^{kf_k(H)} e^{-kf_k(H)} d\mu_{f_k(H)} \\
&=& N_k \int d\mu_{f_k(H)} =N_k. 
\end{eqnarray}
\end{proof} 

According to the pioneering work \cite{Don01}, 
the following provides the ``quantization" of K\"ahler-Einstein metric, 
or more generally, of constant scalar-curvature K\"ahler metric. 
\begin{dfn}
The Hermitian form $H$ is called \textit{balanced} if $b_k(H) = H$ holds. 
\end{dfn} 
As a result of simple unitary diagonalization it is always possible to take a basis $s_i$ 
which is at the same time $H$-orthonormal and $b_k(H)$-orthogonal. 
In this case, the balanced condition is equivalent to $\norm{s_i}_{b_k(H)}=1$. 

In the anti-canonical setting 
\cite{BBEGZ11} intoduced the quantization of the energy functionals 
which we explained in section \ref{Kahler-Ricci flow and optimal degeneration}. 
Once we fix the reference $H_0 \in \cH_k$ the finite-dimensional Monge-Amp\`ere energy 
of a Hermitian form $H \in \cH_k$ is 
\begin{equation}
E_k(H) = -\frac{1}{kN_k}\log \sideset{}{_{H_0}}\det{H}
\end{equation}
and the $D$-energy is quantized to be 
\begin{equation}
D_k= L \circ f_k -E_k. 
\end{equation}
The critical point of the $D_k$-energy is precisely the balanced Hermitian form. 
Moreover, if there exists a unique K\"ahler-Einstein metric $\phi_{\KE}$ 
the image of the $k$-balanced Hermitian form $f_k(H)$ 
converges to $\phi_{\KE}$. 
See \cite{BBGZ13, BBEGZ11} for the proof.

\subsection{Quantization of the K\"ahler-Ricci flow}
\label{Quantization of the Kahler-Ricci flow}

Since $b_k(H)$ is positive Hermitian, one can takes the logarithm matrix
which defines a (not necessarily potitive) Hermitian form.  
Our basic idea is to consider the logarithm $\log{b_k(H)}$
as a counterpart of the Ricci potential. 
 
\begin{lem}\label{L-continuity}
    There exists a positive constant $L$ (independent of $k$) such that  
    \begin{equation}
        \norm{b_k(H)}_{C^2} \leq L \norm{H}_{C^0}
    \end{equation}
    holds for any $H \in \cH_k$. 
    Here we take the norms in the manner of Hilbert-Schmidt and 
    the derivatives in the normal coordinate of the homogeneous space. 
    \end{lem} 
\begin{proof}
    We starts from the expression 
\begin{equation}
    b_k(H)(s_i, s_j) =\int_X h_0(s_i, s_j)e^{-kf_k(H)}d\mu_{f_k(H)}
\end{equation}
 for fixed reference fiber metric $h_0$ and basis $s_i$. 
 We will differentialte it along the fixed geodesic $H_t$. 
 By the standard diagonalization argument there exists 
 $\lambda_1,\lambda_2, \dots, \lambda_{N_k}$ such that 
$e^{\lambda_i t }s_i$ gives an orthonormal basis for $H_t$. 
Adjusting the speed we may assume $\sum \lambda_i^2 \leq 1$. 
It is straightforward to see that 
\begin{align}
    &\frac{d}{dt}\bigg\vert_{t=0}b_k(H)(s_i, s_j) 
    =\frac{d}{dt}\bigg\vert_{t=0}\bigg\{\frac{\int_X h_0(s_i, s_j)e^{-(k+1)f_k(H)}d\mu_0}{\int_X e^{-f_k(H)}d\mu_0}
      \bigg\} \\
    &=\int_X h_0(s_i, s_j) \frac{k+1}{k} 
    \frac{\sum \lambda_i \abs{s_i}^2 }{\sum \abs{s_i}^2} e^{-kf_k(H)} 
    d\mu_{f_k(H)}
    +\int_X h_0(s_i, s_j) e^{-kf_k(H)} d\mu_{f_k(H)} 
    \int_X \frac{\sum \lambda_i \abs{s_i}^2 }{k\sum \abs{s_i}^2} d\mu_{f_k(H)}
\end{align}
so the left-hand side is bounded by 
$3b_k(H)(s_i, s_j)$. 
Similarly, the second derivative at $t=0$ is bounded by 
$9b_k(H)(s_i, s_j)$. 
We conclude the proof since $\norm{{b_k(H)}}_{C^0} = \norm{H}_{C^0}$ by 
Corollay \ref{normalization lemma}. 
\end{proof}
We now define the quantization of the K\"ahler-Ricci flow. 
\begin{dfn}
 \emph{Quantized K\"ahler-Ricci flow} is 
 the evolution of the positive Hermitian form $H=H_t$ 
 which satisfies 
\begin{equation}\label{quantized flow}
\frac{1}{k}\frac{d}{dt}\log{H} = \log{b_k(H)}-\log{H}
\end{equation} 
 for $t>0$. 
\end{dfn} 
If one identifies each tangent vector on $\cH_k=\GL(N_k: \C)/U(N_k)$ 
with a Hermitian form, 
the equation (\ref{quantized flow}) in fact defines a vector field on $\cH_k$. 
By Lemma \ref{L-continuity} 
and the completeness of the Riemannian metric space $\cH_k$,  
 the quantized flow has unique long-time solution 
 for any initial form $H_0$. 
\begin{lem}\label{orthonormal-orthogonal}
If $s_i=s_i(t)$ is a $H$-orthonormal and $b_k(H)$-orthogonal basis 
of $H^0(X, kL)$, 
there exists a skew-Hermitian matrix $X_{ij}(t)$ such that 
 \begin{equation}\label{quantized flow 2} 
    \dot{s}_i =\sum_{j} (-k\log{\norm{s_i}_{b_k(H)}^2+}X_{ij})s_j.
    \end{equation} 
\end{lem}
\begin{proof}
 A standard unitary diagonalization algorithm implies that $s_i(t)$ is smooth in $t$. 
 If we write $s_i =\sum a_{ij}s_j$
    \begin{align*}
    H(\dot{s_i}, s_j) +H(s_i, \dot{s_j}) 
    &=H(s_i, s_j)^{\cdot} -\dot{H}(s_i, s_j)\\
    &= -k \log{\norm{s_i}_{b_k(H)}^2} \cdot \d_{ij}
    \end{align*}
    Hence $a_{ij}=-k \log{\norm{s_i}_{b_k(H)}^2} \d_{ij}$ modulo the skew-Hermitian form.  
\end{proof} 
In the above special choice of the basis, Lemma \ref{normalization lemma} states that 
$\log \norm{s_i}^2_{b_k(H)}$ automatically satisfies the normalized condition 
\begin{equation}
\frac{1}{N_k}\sum (e^{\log \norm{s_i}^2_{b_k(H)}} -1) =0 
\end{equation}
similar to (\ref{Ricci potential}). 
Recall that it is simply due to the choice of the measure $d\mu_\phi$ 
in the definition of the map $p_k$. 
If one istead uses unnormalized measure $e^{-\phi}d\mu_0$ in the definition of $p_k$, 
the equation (\ref{quantized flow}) quantizes the unnormalized K\"ahler-Ricci flow 
\begin{equation}
\frac{\partial}{\partial t} \omega =-\Ric{\omega}. 
\end{equation}
This point was already observed by \cite{Berm13} in the context 
of his \emph{Bergman iteration}.  

\begin{rem}
Let us fix $t>0$ and a geodesic $H'_u$ which tangents to the flow $H_t$ at $u=0$. 
One can take a suitable orthonormal basis $s'_i$ of $H'_0=H_t$ such that 
\begin{equation}
e^{\frac{\lambda_i(t)}{2} u} s'_i \ \ \ (1 \leq i \leq N_k)
\end{equation}
gives an orthonormal basis of $H'_u$. 
If $s'_i$ is further $b_k(H)$-orthogonal, the equation (\ref{quantized flow 2}) 
is written to a simple form 
\begin{equation}
    \frac{\lambda_i}{k} = -\log{\norm{s_i}_{b_k(H)}^2}.  
\end{equation} 
However, this is not always the case. 
\end{rem}

As we mentioned in the previous paragraph, \cite{Berm13} studied another quantization 
of the K\"ahler-Ricci flow. 
It is taking the Bergman iteration 
\begin{equation}\label{Bergman iteration}
    \b_k^j(\phi):=(\overbrace{\b_k \circ \b_k \circ \cdots \circ \b_k}^{j})(\phi) 
\end{equation}
for a given $\phi \in \cH$. 
We may apply Euler's method to derive the difference equation

\begin{equation}\label{difference equation}
\log{H^{(j+1)}}-\log{H^{(j)}} 
= (t_{j+1}-t_j) k (\log{b_k(H^{(j)})} -\log{H^{(j)}})
\end{equation}
from (\ref{quantized flow}). 
Let us choose the equal interval $h=1/k$ and set $t_j=jh$. 
It then yields $H^{(j+1)}=b_k(H^{(j)})$. 
In other words, our quantized flow interporates the Bergman iteration. 
The general convergence reslut states that 
the solution of the associated difference equation  
converges to the solution of the original differential equation. 
In the present setting the space $\cH_k$ itself varies, however, 
we may still obtain the satisfactory uniform estimate. 

\begin{thm}\label{convergence of Euler's method}
    Let $H_t$ be the solution of \ref{quantized flow} 
    starting from $H_0$. 
    For a fixed $T>0$, there exists a constant $C_T$ independent of $k$ such that 
\begin{equation}
\abs{H_{j/k}-b_k^j(H_0)}  \leq C_Tk^{-1}
\end{equation}
holds for any $j/k \leq T$. 
\end{thm}
\begin{proof}
The proof repeats the general convergence result for 
the Euler method. 
Using Taylor's formula we observe 
\begin{align}
    &\log{H_{(j+1)/k}}-\log{H^{(j+1)}}\\
    &=(\log{H_{j/k}}+k^{-1}H_{j/k}^{-1}\dot{H}_{j/k}+O(k^{-2}))
    -\log{H^{(j)}}-(\log{b_k(H^{(j)})-\log{H^{(j)}}}) \\
    &=\log{H_{j/k}}+(\log{b_k(H_{j/k})}-\log{H_{j/k}})+O(k^{-2})
    -\log{H^{(j)}}-(\log{b_k(H^{(j)})-\log{H^{(j)}}}) 
\end{align}
hence it follows 
\begin{align}
    &\abs{\log{H_{(j+1)/k}}-\log{H^{(j+1)}}}\\
    &\leq \abs{\log{H_{j/k}}-\log{H^{(j)}}}
    +\abs{(\log{b_k(H_{j/k})}-\log{H_{j/k}})
    -(\log{b_k(H^{(j)})-\log{H^{(j)}}})}+L_2k^{-2}\\
    & \leq (1+L_1 k^{-1})\abs{\log{H_{j/k}}-\log{H^{(j)}}} +L_2k^{-2}. 
\end{align} 
In the above $L_1$ is the Lipschitz constant and 
$L_2$ is determined by the remainder term of Taylor's formula. 
The both constants are independent of $k$ by Lemma \ref{L-continuity}. 
By induction we deduce 
\begin{align}
    \abs{\log{H_{j/k}}-\log{H^{(j)}}}
    \leq L_2k^{-2}\sum_{i=0}^j(1+L_1k^{-1})^{i}
    \leq L_2\frac{e^{L_1T}}{L_1}k^{-1}
\end{align}
which concludes the proof.  

\end{proof}

As Berman proved, the Bergman iteration approximates the K\"ahler-Ricci flow 
in the following sense.  

\begin{thm}[\cite{Berm13} Theorem 4.18]\label{Berman}
    Let $\phi_t$ be the normalized K\"ahler-Ricci flow (\ref{KRF}) 
    and $\b_k^j(\phi_0)$ be the Bergman iteration for the initial metric $\phi_0$, 
    defined in (\ref{Bergman iteration}). 
    Then for any $T>0$ there exists a constant $C_T$ such that 
    \begin{equation}
    \abs{\phi_{j/k} -\b_k^j(\phi_0)} \leq C_T j k^{-2}
    \end{equation}
    holds for any $j/k \leq T$. 
\end{thm} 

Combining the result with Theorem \ref{convergence of Euler's method}, we may show that 
the quantized flow as well approaches the K\"ahler-Ricci flow. 

\begin{thm}\label{quantized flow vs KRF}
Let $\phi_t$ be the normalized K\"ahler-Ricci flow (\ref{KRF}) 
and $H=H_t$ be the quantized flow (\ref{quantized flow}) initiated from $H_0=p_k(\phi_0)$. 
Then for any $T>0$ there exists a constant $C_T$ such that 
\begin{equation}
\abs{\phi_{(j+1)/k} -f_k(H_{j/k})} \leq C_T k^{-1}
\end{equation}
holds for any $(j+1)/k \leq T$. 
\end{thm}

\begin{proof}
    The triangle inequality yields 
    \begin{align}
        \abs{\phi_{(j+1)/k}-f_k(H_{j/k})}
        \leq \abs{\phi_{(j+1)/k}- \b_k^{j+1}(\phi_0)}
        +\abs{\b_k^{j+1}(\phi_0)-f_k(H_{j/k}) }. 
    \end{align}
    The first term can be estimated by Theorem \ref{Berman}. 
    The second term $\abs{f_k(b_k^j(H_0))-f_k(H_{j/k}) }$
    is estimated by Theorem \ref{convergence of Euler's method}. 
    Note the map $H \mapsto f_k(H)$ is continuous 
    with Lipschitz constant not greater than one, 
    as one can see from the argument in Lemma \ref{L-continuity}.  
\end{proof}

\subsection{Quantized entropy functional}
We will also quantize the entropy functional defined in (\ref{entropy}). 
It is equivalent to the relative entropy in the quantum information theory. 
\begin{dfn}
For two Hermitian forms $A, B$ quantum relative entropy 
is defined to be 
\begin{equation}
    S(A \vert B) 
    := \Tr A(\log{A} -\log{B}). 
\end{equation}
We define the quantized entropy of $H \in \cH_k$ by 
\begin{equation}\label{equatntized entropy}
S_k(H) =  S(b_k(H) \vert H). 
\end{equation} 
From the general theory of quantum entropy, $S_k$ is nonnegative. 
Moreover, $S_k(H)=0$ precisely when $H$ is balanced. 
\end{dfn} 

In our point of energy-theoretic view 
it is more convenient to rewrite 
the entropy into the convex conjugate form, 
as an anologue of (\ref{convex conjugate form0}). 

\begin{lem}\label{convex conjugate form}
    The quantized entropy of $H \in \cH_k$ by the convex conjugate form 
    \begin{equation}
    S_k(H) = \sup_{\lambda_i \in \R}
     \bigg\{ \sum -\frac{\lambda_i}{k} \frac{\norm{s_i}^2_{b_k(H)}}{N_k}
      - \log \frac{1}{N_k}\sum e^{-\frac{\lambda_i}{k} }\bigg\}, 
    \end{equation} 
    where $s_i$ is chosen $H$-orthonormal and $b_k(H)$-orthogonal. 
\end{lem}
\begin{proof}
As an easy application of Lagrange multiplier method 
the supremum is attained by 
\begin{equation}
 \log{N_k} +\sum \frac{\norm{s_i}^2_{b_k(H)}}{N_k} 
 \log{\frac{\norm{s_i}^2_{b_k(H)}}{N_k}} 
 \end{equation}
 which is equivalent to $S_k(H)$. 
\end{proof}

\begin{prop}\label{S bounds minimal slope}
    Quantized entropy bounds the minimal subtracted slope of the L-functional 
    along the geodesics. Namely, 
    \begin{equation}
        S_k(H_0)\leq   \sup_{H_t}\bigg\{
         -\frac{d}{dt}\bigg\vert_{t=0} L\circ f_k(H_t)
         -\log\frac{1}{N_k}\sum e^{-\frac{\lambda_i}{k}}
         \bigg\}
    \end{equation}
    where $H_t$ is arbitrary geodesic emanating from $H_0$. 
\end{prop}

\begin{proof}
    Let $s_i$ be the basis which is $H_0$-orthonormal. 
    Take an arbitrary geodesic $H_t$ and $\lambda_i$ so that  
    $e^{\frac{\lambda_i}{2} t} s_i$ are orthonormal for $H_t$. 
    We compute the differential as 
    \begin{align}
        \frac{d}{dt}L\circ f_k(H_t)
        &=\frac{d}{dt}\bigg\{
             -\log \frac{1}{V}\int_X e^{-f_k(H)}d\mu_0
             \bigg\}\\
        &= \int_X \frac{\lambda_i \abs{s_i}^2}{k\sum \abs{s_i}^2}
        d\mu_{f_k(H_0)}\\
     &=\sum_i \frac{\lambda_i}{k} 
     \int_X \abs{s_i}^2 \frac{e^{-kf_k(H_0)}}{N_k} d\mu_{f_k(H_0)} \\
     &= \sum_i \frac{\lambda_i}{k} 
     \frac{\norm{s_i}_{b_k(H_0)}^2}{N_k}. 
    \end{align}  
Now the point is that $s_i$ is not necessarily $b_k(H_0)$-orthogonal. 
If further $s_i$ is $b_k(H_0)$-orthogonal, it holds  
\begin{equation}
S_k(H_0)=
\sum_i \frac{\lambda_i}{k} \frac{\norm{s_i}_{b_k(H_0)}^2}{N_k} 
-\log\frac{1}{N_k}\sum e^{-\frac{\lambda_i}{k}}. 
\end{equation} 
\end{proof} 

\begin{rem}
    If we replace the subtracted term with 
 \begin{equation}
   \frac{d}{dt}E_k(H_t)= \sum \frac{\lambda_i}{k}, 
 \end{equation} 
 then the quantized D-energy 
 $D_k := L \circ f_k -E_k$ gives a similar bound 
 \begin{equation}
    S(H_0)\leq -\frac{d}{dt} D_k(H_t) 
 \end{equation}
 for an arbitrary geodesic $H_t$. 
\end{rem} 

We will show that 
the above supremum is attained by the direction of the quantized flow. 
the subtracted term should be zero by Lemma \ref{normalization lemma} 
if we set $\lambda_i := -k\log{\norm{s_i}_{b_k(H_0)}^2}$. 

\begin{prop}\label{S as slope along the flow}
 Along the quantized flow (\ref{quantized flow}), 
 \begin{align}
    S_k(H) &= 
         -\frac{d}{dt} L\circ f_k(H_t)\\
         &=-\sum \frac{\lambda_i}{k} \frac{\norm{s_i}_{b_k(H)}^2}{N_k}
         - \log{\frac{1}{N_k}\sum e^{-\frac{\lambda_i}{k}}}, 
 \end{align} 
 where $\lambda_i:=-k\log{\norm{s_i}_{b_k(H)}^2}$. 
\end{prop} 
\begin{proof}
Let us take $H_t$-orthonormal and $b_k(H_t)$-orthogonal basis $s_i(t)$. 
Similaly to the previous proposition we compute 
\begin{align}
    \frac{d}{dt}L\circ f_k(H_t)
    &=\frac{d}{dt}\bigg\{
         -\log \frac{1}{V}\int_X e^{-f_k(H)}d\mu_0
         \bigg\}\\
    &= 2 \Re \int_X \frac{\dot{s}_i\bar{s}_i}{k\sum \abs{s_i}^2}
    d\mu_{f_k(H)}\\
 &=\frac{1}{kN_k}\sum_i  
 2\Re \int_X \dot{s}_i\bar{s}_i e^{-kf_k(H)} d\mu_{f_k(H)} \\
 &= \frac{1}{N_k}\sum_i \bigg\{ \norm{s_i}_{b_k(H)}^2\log{\norm{s_i}_{b_k(H)}^2}
 +b_k(H)(X_{ij}s_j, s_i) +b_k(H)(s_i, X_{ij}s_j)\bigg\}\\
 &= \sum_i \frac{\norm{s_i}_{b_k(H)}^2}{N_k}\log{\norm{s_i}_{b_k(H)}^2}\\
 &= S_k(H). 
\end{align} 
\end{proof} 

As an analogue of the K\"ahler-Ricci flow 
It is natural to expext monotonicity of the entropy alog the 
quantized flow. 
However, we have only asymptotic monototonicity. 
This will be not used in the subsequent discussion but we serve the proof 
for the interested readers. 

\begin{prop}\label{asymptotic monotonicity}
We fix $\phi_0 \in \cH$. 
Let $H_t$ be the quantized flow \ref{quantized flow}
starting from $H_0$. 
\begin{align}
    \frac{d}{dt}S_k(H_t) 
    \leq \frac{k+1}{N_k} S(H_t\vert b_k(H_t))^2.  
\end{align}
\end{prop} 
\begin{proof}
Let us for a moment write $B_i := \norm{s_i}_{b_k(H)}^2$
 and $\dot{f}:=\frac{df}{dt}$.  
 Since $\sum B_i =N$ by Lemma \ref{normalization lemma}, 
 it is observed  
\begin{equation}
    \bigg(\sum B_i \log{B_i}\bigg)^{\cdot}
    =\dot{B_i}\log{B_i}. 
\end{equation}
We compute the differential $\dot{B}_i$ as 
\begin{align}
    \dot{B_i} 
    &=\bigg(\frac{\int_X \abs{s_i}^2 e^{-(k+1)f_k(H)}d\mu_0}{\int_X e^{-f_k(H)}d\mu_0}
      \bigg)^{\cdot} \\
    &=2\Re \int_X \dot{s}_i\bar{s}_i e^{-kf_k(H)} 
    d\mu_{f_k(H)}\\
    &\ \ \ \ \ \ \ \ -2\Re \int_X \abs{s_i}^2 \frac{k+1}{k} 
    \frac{\sum \dot{s}_j\bar{s}_j }{\sum \abs{s_j}^2} e^{-kf_k(H)} 
    d\mu_{f_k(H)}\\
    &\ \ \ \ \ \ \ \ +2\Re \int_X \abs{s_i}^2 e^{-kf_k(H)} d\mu_{f_k(H)} 
    \int_X \frac{\sum \dot{s}_j\bar{s}_j }{k\sum \abs{s_j}^2} d\mu_{f_k(H)}. 
\end{align}

Let us multiply $\log{B_i}$ to each line and 
take the summation in $i$. 
By Lemma \ref{orthonormal-orthogonal}, 
the first line is equivalent to 
\begin{equation}
 -2k \sum_i B_i (\log{B_i})^2
 +\sum_i (\sum_{p} B(X_{ip}s_p, s_i)+\sum_{q} B(s_{i}, X_{iq}s_q))\log{B_i}. 
\end{equation} 
Since $s_i$ is $B$-orthogonal 
this is equivalent to $-2k \sum B_i (\log{B_i})^2 \leq 0$ hence it has desirable sign. 
Similarly, the third line yields 
\begin{equation}
-\frac{1}{kN_k} \sum (B_i \log{B_i}) (B_j \log{B_j}) \leq0. 
\end{equation} 
The problem is about the second line which does not have  
the desirable sign. 
Since $\abs{s_i}^2 \leq \sum \abs{s_j}^2$, the term involving 
the skew-Hermitian matrix $X_{jp}$ is bounded from above by 
\begin{equation}
    -\frac{k+1}{k}\sum_{j, p} \int_X X_{jp} s_p \bar{s_j} e^{-kf_k(H)} d\mu_{f_k(H)}
\end{equation} 
which is again zero because $s_i$ are $B$-orthogonal. 
The term involving $\log{B_j}$ yields 
\begin{align}
    \frac{k+1}{N_k} \sum_{i, j}
    \int_X \log{B_i}\log{B_j}\abs{s_i}^2\abs{s_j}^2 
    e^{-2kf_k(H)} d\mu_{f_k(H)}
    \leq \frac{k+1}{N_k} \big(\sum \log{B_i}\big)^2. 
\end{align} 

We have $\sum B_i =N$  by Lemma \ref{normalization lemma} hence 
the inequality of arithmetic and geometric means
$\sum \log{B_i} \leq 0$. 
\end{proof}

%%%%%%%%%%%%%%%%%%%%%%%%%%%%%%%%%%%%%%%%%%%%%%%%%%%%%%%%%%%%%%%%%%%%%%%%%%%%%%%%%%
\section{Convergence results}
\label{Convergence results}
%%%%%%%%%%%%%%%%%%%%%%%%%%%%%%%%%%%%%%%%%%%%%%%%%%%%%%%%%%%%%%%%%%%%%%%%%%%%%%%%%%

In this section We prove Theorem \ref{Theorem A} and \ref{Theorem B}. 

\begin{thm}
    Let $\phi$ be a K\"ahler metric and $p_k(\phi)$ be 
    the associated Hermitian $L^2$-inner product 
    defined on the space of anti-canonical sections. 
    Then we have the convergence of the entropies: 
    \begin{equation}
    \lim_{k \to \infty}S_k(p_k(\phi)) =S(\phi). 
    \end{equation}
\end{thm}

\begin{proof}
    Let us utilize the quantized flow emanating from $H_0:= p_k(\phi)$. 
Recall that Proposition \ref{S as slope along the flow} tells 
    \begin{equation}
        S_k(H_t)=-  \frac{d}{dt} L\circ f_k(H_t). 
    \end{equation}
On the other hand, along the normalized K\"ahler-Ricci flow 
emanating from $\phi$ 
we observe 
\begin{align}
    S(\phi_t)
    &= \int_X \rho_{t} d\mu_{\phi_t} 
    =-\int_X \dot{\phi}_t d\mu_{\phi_t}\\
    &=-\frac{d}{dt}L(\phi_t). 
\end{align} 
Keeping $H_t \sim p_k(\phi_t)$ in mind, it is sufficient to show the 
convergene of the slope 
\begin{equation}
    \frac{d}{dt} L\circ f_k(H_t) \to \frac{d}{dt}L(\phi_t).    
\end{equation}   
Now we apply Theorem \ref{quantized flow vs KRF} in this situation. 
It follows that the energy functional itself has well-approximatition 
\begin{equation}
L\circ f_k(H_t) =  L(H_t) + o(\frac{1}{k}). 
\end{equation}
If we fix $T$, the convergence  
is unifom in $t \in [0, T)$. 
Now the proof is completed as soon as one notice the general fact that 
uniform convergence of continuously differentiable convex functions $g_k \to g$ on the interval 
$[0, T)$
impies $g_k'(0) \to g_k'(0)$. 
For the readear's convenience we surve a proof of this fact. 
Fix any $t>0$ and take arbitrary $\e>0$. 
If $k$ is sufficiently large against $\e t$ we have 
\begin{equation}
    \frac{g(t)-g(0)}{t} \leq  \frac{g_k(t)-g_k(0)}{t} +\frac{\e t}{t}
\end{equation} 
therefore 
\begin{align}
    g'(0) &= \inf_{t>0}  \frac{g(t)-g(0)}{t} \\
    &\leq  \frac{g_k(t)-g_k(0)}{t} +\e 
    =g_k'(s) +\e 
\end{align}
holds for some $s \in [0, t]$. 
It follows $g'(0) \leq \liminf_{k\to \infty} g_k'(0)$. 
The same argument provides the converse inequality. 
\end{proof} 

\begin{cor}
If the anti-canonical polarization $(X, -K_X)$ admits a K\"ahler-Einstein metric 
$\phi_{\KE}$, 
then $p_k(\phi_{\KE})$ is asymptotically balanced in the sense that 
\begin{equation}
\lim_{k \to \infty} S_k(p_k(\phi_{\KE})) =0 
\end{equation}
holds. 
\end{cor} 

In the same situation \cite{BBEGZ11} shows that 
$\cH_k$ admits an anti-canonical balanced form $H_{\bal, k}$ and
 $\lim_{k \to \infty} f_k(H_{\bal, k})=\phi_{\KE}$. 
The above corollary immediately recovers their result. 
Notice that $p_k(\phi_{\KE}) \neq H_{\bal, k}$ in general. 

%%%%%%%%%%%%%%%%%%%%%%%%%%%%%%%%%%%%%%%%%%%%%%%%%%%%%%%%%%%%%%%%%%%%%%%%%%%%%%%%%%
\section{Optimal non-Archimedean norm for the quantized entropy}
%%%%%%%%%%%%%%%%%%%%%%%%%%%%%%%%%%%%%%%%%%%%%%%%%%%%%%%%%%%%%%%%%%%%%%%%%%%%%%%%%%

\subsection{Non-Archimedean norms}
\label{Non-Archimedean norms}
For terminologies about norms and filtrations, we refer \cite{BHJ17}, section $1$. 

Recall that the non-Archimedean norm $\nu$ on the finiite-dimenstional 
vector space $V$ is a function $\nu \colon V \to \R_{\geq 0}$ such that 
\begin{itemize}
\item[$(1)$] \ $\nu(s)=0$ if and only if $s=0$.  
\item[$(2)$] \ $\nu(a \cdot s) =\nu(s)$  holds for any $a \in \C^*$, $s \in V$, and  
\item[$(3)$] \ $\nu(s+s') \leq \max \{ \nu(s), \nu(s')\} $ holds for any $s, s' \in V$. 
\end{itemize}
We denote by $\cH_k^\NA$ the set of non-Archimedean norms on $H^0(X, kL)$. 
The associated filtration 
\begin{equation}\label{filtration}
F^\lambda H^0(X, kL)
=\bigg\{ s \in H^0(X, kL): \nu(s) \leq e^{-\lambda} \bigg\} 
\end{equation}
for $k \in \Z_{\geq 0}, \lambda \in \R$ 
produces weights $\lambda_1, \lambda_2, \dots, \lambda_{N_k}$ and vectors $s_1, s_2, \dots, s_{N_k}$ describing the jump of dimensions. In this way each element of $\cH_k^\NA$ can be seen as a direction pointed to by a geodesic ray on $\cH_k$. 
Let $\cH^\NA$ be the space of K\"ahler metrics on the Berkovich analytification $(X^\NA, L^\NA)$. 
For the background materials bout Berkovich analytifictions see \cite{BHJ17}, section $6$. 
The space $\cH_k^\NA$ gives the quantization of $\cH^\NA$ 
by the non-Archimedean Fubini-Study map 
\begin{equation}
f_k \colon \cH_k^\NA \to \cH^\NA.  
\end{equation}
Basically in the non-Archimedean setting we should take the maximum in place of the square sum so that $f_k$ is defined to be 
\begin{equation}
f_k(\nu) = \max_s \frac{1}{k}\log \frac{\abs{s}^2}{\nu(s)} \in \cH^\NA. 
\end{equation}
Here $\abs{s}$ is a function on $X^\NA$, which satisfies $\abs{s(v)}^2=e^{-v(s)}$ for any quasi-monomial valuation $v \in X^\NA$.  

Any equivalent class of test configuration defines the element of $\cH^\NA$. 
In fact the filtration 
\begin{equation}
F_{(\cX, \cL)}^\lambda H^0(X, kL)
=\bigg\{ s \in H^0(X, kL): \text{ $\tau^{-\lambda}s$ extends holomorphically to $\cX$. } \bigg\} 
\end{equation}
associated with the test configuration $(\cX, \cL)$ defines $\nu_k \in \cH_k^\NA$. 
If denotes the Gauss extension for the valuation $v$ by $G(v)$, one can see that 
\begin{equation}
f_k(\nu_k) (v) = -\frac{1}{k} G(v)(k\cL )
\end{equation} 
holds. 
Here $G(v)(k\cL )$ denotes the 
valuation for a general section of the line bundle $k\cL$. 
The right-hand side is in fact independent of $k$ and hence defines $\nu_{(\cX, \cL)} \in \cH^\NA$. 
Conversely, for any non-Archimdean norm $\nu \in \cH_k^\NA$ 
the image of the Fubini-study map $f_k(\nu)$
 is represented by a test configuration 
 so that we may take $L^\NA(f_k(\nu))$ ;which was defined immediately after  
 Definition \ref{definition of S-stability}.

\begin{dfn}
Let $F_k^\NA(\nu) = -\log \frac{1}{N_k} \sum_i e^{-\frac{\lambda_i}{k}}$.  
Define the non-Archimedean version of the quantized entropy as  
\begin{equation}
S_k^\NA= L^\NA \circ f_k -F_k^\NA. 
\end{equation}
\end{dfn} 

\begin{thm}
We have the quantized version of the equality (\ref{optimal degeneration}) so that   
\begin{equation}
\inf_{H \in \cH_k} S_k(H) = \max_{\nu \in \cH_k^\NA} (-S_k^\NA(\nu))
\end{equation}
holds for any fixed $k \in \N$. 
\end{thm} 

\begin{proof}
 Let us first take any $\nu \in \cH_k^\NA$ to show 
 the one-sided inequality. 
Via the filtration (\ref{filtration}), the norm $\nu$ is reperesened by some basis $s_i$ and weights $\lambda_i$. 
 Along the geodesic $H_t$ such that  
 $e^{\frac{\lambda_i}{2}t}s_i$ is $H_t$-orthonormal, convexity 
 of the $L$-functional 
 implies 
 \begin{align}
   -L^\NA\circ f_k (\nu) &= \lim_{t \to \infty} -\frac{L\circ f_k(H_t)}{t}\\
   &\leq -\frac{d}{dt}\bigg\vert_{t=0} \frac{L\circ f_k(H_t)}{t}
   = -\sum \frac{\lambda_i}{k} \frac{\norm{s_i}_{b_k(H_0)}^2}{N_k}. 
 \end{align}
It follows 
 \begin{equation}
    -S_k^\NA(\nu) \leq 
    -\sum \frac{\lambda_i}{k} \frac{\norm{s_i}_{b_k(H_0)}^2}{N_k}
    - \log{\frac{1}{N_k}\sum e^{-\frac{\lambda_i}{k}}}. 
 \end{equation} 
Notice that $H_t$ is adopted after the filtration and $\lambda_i$ are fixed, 
so in the above computations $s_i$ can be assumed 
$H_0$-orthonoemal and $b_k(H_0)$-orthogonal. 
Indeed the diagonalization procedere can be compatible 
with respect to the filtration $F^\lambda H^0(X, kL)$. 
Hence we conclude $-S_k^\NA(\nu) \leq S_k(H_0)$ 
with Lemma \ref{convex conjugate form}. 

To show the opposite inequality, 
we will exploit the quantized flow $H_t$. 
Let us take a sequence $t_j \to \infty$ and 
the non-Archimedean norm $\nu_j$ represented by 
the $H_{t_j}$-orthonormal basis $s_i$ and the weights 
$\lambda_i= -k\log\norm{s_i}_{b_k(H_{t_j})}^2$. 
The quantized entropy is computed 
by Proposition \ref{S as slope along the flow} as 
\begin{align}
    S_k(H_{t_j}) 
    = -\frac{d}{dt} \bigg\vert_{t=t_j}L\circ f_k(H_t)
&=-\sum \frac{\lambda_i}{k} \frac{\norm{s_i}_{b_k(H_{t_j})}^2}{N_k}
- \log{\frac{1}{N_k}\sum e^{-\frac{\lambda_i}{k}}}\\
&=-S_k^\NA(\nu_j). 
\end{align}
Consequently, we obtain 
$S_k(H_{t_j})\leq -S_k^\NA(\nu_j)$. 
\end{proof} 

\begin{cor}
If there exists a balanced Hermitian inner product 
then $S_k^\NA(\nu)  \geq 0$ for all $\nu \in \cH_k^\NA$. 
\end{cor} 
This is similar to the result of \cite{ST19}. 
They showed that existence of the anti-canonical balanced metric implies 
``F-stability''. 
The F-stability requires for any test configuration 
of exponent $k$ 
positivity of the non-Archimedean Ding functional 
$D_k^\NA:= L^\NA \circ f_k -E_k$ plus the positive constant $q$ 
which reflects multiplicity of the central fiber.


\begin{thebibliography}{99.}%
    % Use the following syntax and markup for your references if 
    % the subject of your book is from the field 
    % "Mathematics, Physics, Statistics, Computer Science"
    %
    % Contribution 
    
    %\bibitem[Bac14]{Bac14}M. Ba\v{c}\'ak: 
    %       \newblock \emph{Convex Analysis and Optimization in Hadamard Spaces}. 
    %       \newblock De Gruyter Series in Nonlinear Analysis and Applications \textbf{22}, 2014. 
    
    %\bibitem[BM85]{BM85}S. Bando, and T. Mabuchi:
    %	\newblock \emph{Uniqueness of K\"ahler-Einstein metrics modulo connected group actions}.
    %	\newblock algebraic geometry, Sendai, 1985, 11--40, Adv. Stud. Pure Math., 10, North Holland, Amsterdam, 1987. 
    
    %\bibitem[BT76]{BT76}E. Bedford and A. Taylor: 
     %\newblock \emph{The Dirichlet problem for a complex Monge--Amp\`{e}re equation}. 
    % \newblock  Invent. Math. \textbf{37} (1976), no. 1, 1--44.
    
    %\bibitem[BM86]{BM86} S. Bando, and T. Mabuchi:
    %	\newblock \emph{On some integral invariants on complex manifolds, I}.
    %	\newblock Proc. Japan Acad. Ser. A Math. Sci., {\bf 62} (1986), no. 5, 197--200. 
    
    %\bibitem[B06]{Berm06}R. Berman: 
    %  \newblock {\em Super Toeplitz operators on line bundles}. 
    %  \newblock J. Geom. Anal. \textbf{16} (2006), no. 1, 1--22.
    
    %\bibitem[B09]{Berm09}R. Berman:
    %   \newblock {\em Bergman kernels and equilibrium measures      
    %                                  for line bundles over projective manifolds}. 
    %   \newblock  Amer. J. Math. \textbf{131} (2009), no. 5, 1485--1524. 
    
    \bibitem{Berm13}R. Berman:  
       \newblock {\em Relative K\"ahler-Ricci flows and their quantization.}
       \newblock Anal. PDE. \textbf{6} (2013), 131--180. 
       
    \bibitem[B13]{Berm13b}R.~J. Berman: 
       \newblock \emph{A thermodynamical formalism for Monge-Amp\`ere equations, Moser-Trudinger inequalities and K\"ahler-Einstein metrics}. 
       \newblock  Adv. Math. \textbf{248} (2013), 1254--1297. 	
    
    %\bibitem[B14]{Berm14}R.~J. Berman: 
    %    \newblock \emph{On the optimal regularity of weak geodesics in the space of metrics on a polarized manifold}. 
    %    \newblock \texttt{arXiv:1405.6482}. 
    
    \bibitem{Berm16}R.~J. Berman: 
       \newblock \emph{K-polystability of Q-Fano varieties admitting Kahler--Einstein metrics}. 
       \newblock  Invent. Math. \textbf{203} (2016), no. 3, 973--1025. 
    
    \bibitem[B18]{Berm18}R.~J. Berman: 
       \newblock \emph{K\"ahler-Einstein metrics, canonical random point processes and birational geometry}. 
       \newblock Algebraic geometry: Salt Lake City 2015, 29--73, Proc. Sympos. Pure Math., 97.1, Amer. Math. Soc., Providence, RI, 2018. 
    
    \bibitem[B20]{Berm20}R.~J. Berman: 
       \newblock \emph{An invitation to K\"ahler-Einstein metrics and random point processes}. 
       \newblock To appear in Surveys in Differential Geometry. 

    %\bibitem[BB17]{BB17}R.~J. Berman and B. Berndtsson: 
    %    \newblock \emph{Convexity of the K-energy on the space of Kahler metrics and uniqueness of extremal metrics}. 
    %    \newblock J. Amer. Math. Soc. \textbf{30} (2017), 1165--1196. 
    
    %\bibitem[B11]{Bern11}B. Berndtsson: 
     %   \newblock \emph{A Brunn-Minkowski type inequality for Fano manifolds and some uniqueness theorems in K\"ahler geometry.}   
    %    \newblock Invent. Math. {\bf 200} (2015), no. 1, 149--200.  
    
    %\bibitem[B18]{Bou18}S. Boucksom: 
    %\newblock \emph{Variational and non-Archimedean aspects of the Yau-Tian-Donaldson conjecture}.
    %\newblock \texttt{arXiv:1805.03289}. 
    
    \bibitem{BBGZ13}R.~J. Berman, S. Boucksom, V. Guedj, and A. Zeriahi: 
         \newblock \emph{A variational approach to complex Monge-Amp\`ere equations}. 
         \newblock Publ. Math. Inst. Hautes \'Etudes Sci. \textbf{117} (2013), 179--245.
    
    \bibitem{BBEGZ11}R.~J. Berman, S. Boucksom, P. Eyssidieux, V. Guedj, and A. Zeriahi: 
       \newblock \emph{K\"ahler-Einstein metrics and the K\"ahler-Ricci flow on log Fano varieties}. 
       \newblock to appear in J. Reine Angew. Math., \texttt{arXiv:1111.7158}.  
    
    \bibitem{BBJ15}R.~J. Berman, S. Boucksom, and M. Jonsson: 
        \newblock \emph{A variational approach to the Yau-Tian-Donaldson conjecture}. 
        \newblock J. Amer. Math. Soc. \textbf{34} (2021), 605--652
        
    %\bibitem[BDL15]{BDL15}R.~J. Berman, T. Darvas, and C.~H. Lu: 
    %    \newblock \emph{Convexity of the extended K-energy and the large time behavior of the weak Calabi flow}. 
    %    \newblock Geometry and Topology, \textbf{24}, 1907--1967 (2020)      
    
    %\bibitem[BWN14]{BWN14}R.~Berman and D.~Witt Nystr\"om: 
    %\newblock \emph{Complex optimal transport and the pluripotential theory of K\"ahler-Ricci solitons.}
    %\newblock \texttt{arXiv:1401.8264}. 
    
    %\bibitem[BEGZ10]{BEGZ10}S.~Boucksom, P.~Eyssidieux, V.~Guedj, and A.~Zeriahi: 
    %   \newblock \emph{Monge--Amp\`{e}re equations in big cohomology classes.}
    %   \newblock   Acta Math. \textbf{205} (2010), no. 2, 199--262. 
    
    %\bibitem[BEG13]{BDG13}S.~Boucksom, P.~Eyssidieux and V.~Guedj: 
    %     \newblock{An introduction to the K\"ahler-Ricci flow}. 
    %     \newblock Lecture Notes in Math. \textbf{2086}, Springer 2013. 
    
    %\bibitem[BFJ16]{BFJ16}
    %S. Boucksom, C. Favre and M. Jonsson: 
    %\newblock \emph{Singular semipositive metrics in non-Archimedean geometry}.
    %\newblock J. Algebraic Geom. \textbf{25} (2016), 77--139. 
    
    \bibitem{BHJ17}
    S. Boucksom T. Hisamoto and M. Jonsson: 
    \newblock \emph{Uniform K-stability, Duistermaat-Heckman measures and singularities of pairs}. 
    \newblock Ann. Inst. Fourier (Grenoble) \textbf{67} no. 2 (2017), 743--841. 
    
    \bibitem{BHJ19}
    S. Boucksom T. Hisamoto and M. Jonsson: 
    \newblock \emph{Uniform K-stability and asymptotics of energy functionals in K\"ahler geometry}. 
    \newblock J. Eur. Math. Soc.. \textbf{21} no. 9 (2019), 2905--2944
    
    %\bibitem[B16]{Bir16}C. Birkar: 
    %    \newblock \emph{Singularities of linear systems and boundedness of Fano varieties}. 
    %    \newblock  Ann. of Math. (2) \textbf{193}, 347-405 (2021)
    
    %\bibitem[BJ17]{BJ17}H.~Blum and M.~Jonsson: 
    %\newblock \emph{Thresholds, valuations, and K-stability}.
    %\newblock \texttt{arXiv:1706.04548v1}. 
    
    %\bibitem[BJ18]{BJ18}S. Boucksom and M. Jonsson: 
    %\newblock \emph{A non-Archimedean approach to K-stability}.
    %\newblock \texttt{arXiv:1805.11160}. 
    
    \bibitem{BLXZ21}H.~Blum, Y.~Liu, C.~Xu, and Z.~Zhuang: 
     \newblock \emph{The existence of the K\"ahler-Ricci soliton degeneration.}
     \newblock to appear in Forum of Mathematics, Pi. 
    
    \bibitem{BLZ19}H.~Blum, Y.~Liu, and C.~Zhou: 
     \newblock \emph{Optimal destabilization of K-unstable Fano varieties via stability thresholds.}
     \newblock to appear in Geom. Topol. \texttt{arXiv:1907.05399}
    
    \bibitem{Cao85}H.D.Cao: 
     \newblock \emph{Deformation of K\"ahler metrics to K\"ahler-Einstein metrics on compact K\"ahler manifolds.}
     \newblock Invent. Math. \textbf{81} (1985), no. 2, 359--372.
     
    %\bibitem[C00a]{Chen00a}X. Chen:
    %	\newblock \emph{The space of K\"ahler metrics}.
    %	\newblock J. Differential Geom. {\bf 56} (2000), 189--234. 
    
    \bibitem[C00b]{Chen00b}X. Chen: 
            \newblock \emph{On the lower bound of the Mabuchi energy and its application}.                    
            \newblock Int. Math. Res. Not. {\bf 4} (2000), no. 12, 607--623. 
    
    \bibitem{CDS15}X. Chen, S.~K. Donaldson and S. Sun: 
    \newblock\emph{Kahler-Einstein metrics on Fano manifolds, I, II, III}.
    \newblock J. Amer. Math. Soc. \textbf{28} (2015), 183--278. 
    
    \bibitem{CHT17}T.~C. Collins, T. Hisamoto, R. Takahashi: 
     \newblock\emph{The inverse Monge-Amp\`ere flow and applications to K\"ahler-Einstein metrics}. 
     \newblock J. Differential Geom. \textbf{120} (2022), 51--95. 
    
    %\bibitem[CRZ18]{CRZ18}I.~A. Cheltsov, Y.~A. Rubinstein and K. Zhang: 
    %  \newblock \emph{Basis log canonical thresholds, local intersection estimates, and asymptotically log del Pezzo surfaces}. 
    %  \newblock Selecta Mathematica \textbf{25}(34),  (2019) 
    
    \bibitem{CSW15}X. Chen, S. Sun, and B. Wang: 
       \newblock\emph{K\"ahler-Ricci flow, K\"ahler-Einstein metric, and K-stability}.
        \newblock Geom. Topol. \textbf{22} (2018), no. 6, 3145--3173.
    
    %\bibitem[CW14]{CW14}X. Chen and B. Wang: 
    %   \newblock\emph{Space of Ricci flows (II)}. 
     %  \newblock J. Differential Geom. \textbf{116}(1), 1--123 (2020).
    
    %\bibitem[CT08]{CT08}X. Chen and Y. Tang: 
    %    \newblock \emph{ Test configuration and geodesic rays}. 
    %    \newblock Ast\'{e}risque No. \textbf{321} (2008), 139--167.
    
    \bibitem{CTW17}J. Chu, V. Tossatti and B. Weikove: 
        \newblock \emph{ On the Regularity of Geodesics in the Space of K\"ahler Metrics}. 
        \newblock Annals od PDE \textbf{3} (2017) no. 2.  
    
    \bibitem{CTW18}J. Chu, V. Tossatti and B. Weikove: 
        \newblock \emph{ $C^{1,1}$-regularity for degenerate complex Monge-Amp\`ere equations and geodesic rays}. 
        \newblock Communications in Partual Differential Equations \textbf{43} (2018), no. 2, 292--312. 
    
    \bibitem{Dar15}T. Darvas: 
        \newblock \emph{The Mabuchi geometry of finite energy classes}. 
        \newblock  Adv. Math. \textbf{285} (2015), 182--219.
    
    \bibitem{Dar17a}T. Darvas: 
        \newblock \emph{Weak geodesic rays in the space of K\"ahler potentials and the class $\cE(X,\omega)$}. 
        \newblock  J. Inst. Math. Jussieu \textbf{16} (2017), no. 4, 837--858. 
    
    \bibitem{Dar17b}T.~Darvas: 
        \newblock \emph{The Mabuchi completion of the space of K\"ahler potentials}. 
        \newblock  Amer. J. Math. \textbf{139} (2017), no. 5, 1275--1313. 
    
    \bibitem{DH17}T. Darvas and W. He: 
        \newblock \emph{Geodesic rays and K\"ahler-Ricci trajectories on Fano manifolds}. 
        \newblock Trans. Amer. Math. Soc. \textbf{369} (2017), no. 7, 5069--5085.
    
    %\bibitem[DL18]{DL18}T. Darvas and C.~H. Lu: 
    %    \newblock \emph{Uniform convexity in $L^p$ Mabuchi geometry, the space of rays, and geodesic stability}. 
    %    \newblock \texttt{arXiv:1810.04661}. 
    
    %\bibitem[DLR18]{DLR18}T. Darvas, C.~H. Lu, and Y.~A. Rubinstein: 
    %    \newblock \emph{Quantization in geometric pluripotential theory}. 
    %    \newblock \texttt{arXiv:1806.03800}. 
        
    \bibitem{DS16}V. Datar and G. Sz\'ekelyhidi: 
       \newblock \emph{K\"ahler-Einstein metrics along the smooth continuity method}.  
       \newblock Geom. Funct. Anal. \textbf{26} (2016), no. 4, 975--1010.    
    
    %\bibitem[D92]{Dem92}J.~P. Demailly: 
     %  \newblock {\em Regularization of closed positive currents and intersection theory. }
    %   \newblock  J. Algebraic Geom. \textbf{1} (1992), no. 3, 361--409.
    
    %\bibitem[Der14a]{Der14a} 
    %R.~Dervan: 
    %\newblock \emph{Uniform stability of twisted constant scalar curvature K\"ahler metrics}. 
    %\newblock \texttt{arXiv:1412.0648}. 
    %\newblock To appear in Int. Math. Res. Notices. 
    
    %\bibitem[Der14b]{Der14b}
    %R.~Dervan: 
    %\newblock \emph{Alpha invariants and coercivity of the Mabuchi functional on Fano manifolds.}
    %\newblock Int. Math. Res. Notices \textbf{26} (2015), 7162--7189.  
    
    \bibitem{DS17}R. Dervan, G. Sz\'ekelyhidi: 
           \newblock \emph{The K\"ahler-Ricci flow and optimal degenerations}. 
           \newblock J. Differential Geom. \textbf{116} (2020), 187--203. 
           
    \bibitem{Din88} W. Ding:
        \newblock \emph{Remarks on the existence problem of positive K\"ahler-Einstein metrics}.
        \newblock Math. Ann. {\bf 282} (1988), no. 3, 463--471.
     
    %\bibitem[DT92a]{DT92a}W. Ding and G. Tian: 
     %      \newblock \emph{K\"ahler-Einstein metrics and the generalized Futaki invariant.}
    %       \newblock Invent. Math. \textbf{110} (1992), no. 2, 315--335. 
     
    %\bibitem[DT92b]{DT92b}W. Ding and G. Tian: 
    %        \newblock \emph{The generalized Moser-Trudinger inequality.}
     %       \newblock Nonlinear Analysis and Microlocal Analysis: Proceedings of the International Conference at Nankai Institute of Mathematics. World Scientific, 57--70, 1992.
     
    \bibitem{Don01}S.~Donaldson: 
    \newblock \emph{Scalar curvature and projective embeddings, I}. 
    \newblock J. Differential Geom. {\bf 59} (2001) 479--522. 
     
    \bibitem{Don02}S.~K. Donaldson: 
       \newblock \emph{Scalar curvature and stability of toric varieties}. 
       \newblock  J. Differential Geom. \textbf{62} (2002), no. 2, 289--349.  
     
    \bibitem{Don05}S.~K. Donaldson: 
       \newblock \emph{Lower bounds on the Calabi functional}. 
       \newblock  J. Differential Geom. \textbf{70} (2005), no. 3, 453--472.
    
    %\bibitem[D15]{Don15}S.~K. Donaldson: 
     %      \newblock \emph{The Ding functional, Berndtsson convexity and moment maps}. 
    %       \newblock \texttt{arXiv:1503.05173v1}.    
    
    %\bibitem[FO16]{FO16}K. Fujita and Y. Odaka: 
    %    \newblock \emph{On the K-stability of Fano varieties and anticanonical divisors}. 
     %   \newblock Tohoku Math. J. \textbf{70}(4) 70(4), 511--521 (2018). 
    
    \bibitem{FZ19}P. Flurin and S. Zelditch: 
        \newblock \emph{Entropy of Bergman measures of a toric K\"ahler manifold}. 
        \newblock Pure Appl. Math. Q. \textbf{18} (2022), 269--303. 

    \bibitem{Fut83} A. Futaki:
        \newblock \emph{An obstruction to the existence of Einstein K\"ahler metrics}.
        \newblock Invent. Math. {\bf 73} (1983), no. 3, 437--443.
      
    \bibitem{GZ17} V. Guedj et al.: 
        \newblock \emph{Complex Monge-Amp\`ere equations and geodesics in the space of K\"ahler metrics}.
        \newblock EMS Tracts in Mathematics. {\bf 26} (2017).      
      
    %\bibitem[GLZ18]{GLZ18} V.~Guedj, H. C.~Lu, and A.~Zeriahi: 
    %	\newblock \emph{Pluripotential K\"ahler-Ricci flows}.
    %	\newblock \texttt{arXiv:1810.02121}.    
      
    \bibitem{Guedj12} V. Guedj and A. Zeriahi: 
        \newblock \emph{Degenerate complex Monge-Amp\`ere equations}.
        \newblock Lecture Notes in Mathematics. \textbf{2038} (2012).    
    
    %\bibitem[GPSS09]{GPSS09}B.~Guo, D.~H. Phong, J.~Song, and J.~Sturm:  
    %    \newblock \emph{Compactness of K\"ahler-Ricci solitons on Fano manifolds}. 
    %    \newblock \texttt{arXiv:1805.03084v1}.  
      
    \bibitem{HL20}J.~Han and C.~Li
     \newblock \emph{Algebraic uniqueness of K\"ahler-Ricci flow limits and optimal degenerations of Fano varieties.}
     \newblock to appear in Geometry and Topology.  
      
    \bibitem{He16}W. He: 
            \newblock \emph{K\"ahler-Ricci soliton and H-functional.}  
            \newblock Asian J. Math. \textbf{20} (2016), 645--664. 
      
    \bibitem{His16}T.~Hisamoto: 
       \newblock \emph{On the limit of spectral measures associated to a test configuration of a polarized K\"{a}hler manifold}.  
       \newblock J. Reine Angew. Math. \textbf{713} (2016), 129--148. 
    
    \bibitem{His19}T.~Hisamoto: 
     \newblock \emph{Geometric flow, Multiplier ideal sheaves and Optimal destabilizer for a Fano manifold.}
     \newblock to appear in the Journal of Geometric Analysis. 
    
    \bibitem{His23}T.~Hisamoto: 
     \newblock \emph{Quantization of the K\"ahler-Ricci flow and the optimal destabilizer for a Fano manifold.}
     \newblock in preparation. 
    
    
    %\bibitem[I18]{I18}E.~Inoue: 
    %    \newblock \emph{The moduli space of Fano manifolds with K\"ahler-Ricci solitons}. 
    %    \newblock \texttt{arXiv:1802.08128v3}. 
    
    \bibitem{IOOS22}L. IOOS: 
    \newblock \emph{Balanced metrics for K\"ahler-Ricci solitons and 
    quantized Futaki invariants}. 
    \newblock J. Funct. Anal. \textbf{282} (2022) no. 8, 109400

    %\bibitem[KP08]{KP08}W.~A. Kirk and B.~Panyanak:
    %    \newblock \emph{A concept of convergence in geodesic spaces}. 
    %    \newblock Nonlinear Analysis: Theory, Methods \& Applications \textbf{68} (2008), 3689--3696. 
    
    %\bibitem[KMM92]{KMM92}J. Koll\'{a}r, Y. Miyaoka, and S. Mori: 
    %  \newblock \emph{Rational connectedness and boundedness of Fano manifolds.}
    %  \newblock J. Differential Geom. \textbf{36} (1992), no. 3, 765--779.
    
    \bibitem{LX14}C. Li and C. Xu: 
      \newblock \emph{Special test configuration and K-stability of Fano varieties.}
      \newblock Ann. of Math. (2) \textbf{180} (2014), no. 1, 197--232.
    
    \bibitem{LXZ21}Y.~Liu, C.~Xu, and Z.~Zhuang: 
      \newblock \emph{Finite generation for valuations computing stability thresholds and applications to K-stability.}
      \newblock Ann. of Math. \textbf{196} (2022), 507--566. 
    %\bibitem[L17]{Li17}C. Li: 
    %  \newblock \emph{Yau-Tian-Donaldson correspondence for K-semistable Fano manifolds}
    %  \newblock J. reine angew. Math. \textbf{733} (2017), 55--85.
    
    \bibitem{Mab86}T.~Mabuchi: 
     \newblock \emph{K-energy maps integrating Futaki invariants.}
     \newblock Tohoku Math. J. (2) \textbf{38} (1986), no. 4, 575--593. 
    
    %\bibitem[Mab01]{Mab01}T.~Mabuchi: 
    %      \newblock \emph{K\"ahler-Einstein metrics for manifolds with nonvanishing Futaki character}. 
    %      \newblock Tohoku Math. J. (2) \textbf{53} (2001), no. 2, 171--182.
    
    \bibitem{N90}A.~M. Nadel: 
       \newblock \emph{Multiplier Ideal Sheaves and K\"ahler-Einstein Metrics of Positive Scalar Curvature}. 
       \newblock Ann. of Math. \textbf{132} (1990), no. 3, 549--596. 
       
    %\bibitem[OT87]{OT87}T. Ohsawa and K. Takegoshi: 
    %   \newblock \emph{ On the extension of $L^2$ holomorphic functions}. 
    %   \newblock Math. Z. \textbf{195} (1987), no. 2, 197--204. 
    
    \bibitem{Pali08}N. Pali: 
        \newblock \emph{Characterization of Einstein-Fano manifolds via the K\"ahler-Ricci flow}. 
        \newblock Indiana Univ. Math. J., \textbf{57} (2008), no. 7, 3241--3274. 
    
    \bibitem{Pere02}G. Perelman: 
        \newblock \emph{The entropy formula for the Ricci flow and its geometric applications.}
        \newblock \texttt{arXiv:math/0211159}. 
    
    %\bibitem[PS06]{PS06} D. H. Phong, and J. Sturm:
    %	\newblock \emph{The Monge-Amp\`ere operator and geodesics in the space of K\"ahler potentials}
    %	\newblock Invent. Math. {\bf 166} (2006), no. 1, 125--149.
        
    %\bibitem[PSS06]{PSS06}D.~H. Phong, N. Sesum and J. Sturm: 
     %  \newblock \emph{Multiplier Ideal Sheaves and the K\"ahler-Ricci Flow}.  
     %  \newblock \texttt{arXiv:math/0611794}. 
    
    \bibitem{PSSW09}D.~H. Phong, J. Song, J. Sturm and B. Weinkove: 
        \newblock \emph{The K\"ahler-Ricci flow and the $\bar{\partial}$-operator on vector fields}. 
        \newblock J. Differential Geom. \textbf{81} (2009), 631--647. 
        
    %\bibitem[PS07]{PS07}D.~H. Phong and J. Sturm: 
    %   \newblock \emph{Test configurations for K-stability and geodesic rays}. 
     %  \newblock J. Symplectic Geom. \textbf{5} (2007), no. 2, 221--247.
    
    %\bibitem[PS10]{PS10}D.~H. Phong and J. Sturm:
     %   \newblock \emph{Regularity of geodesic rays and Monge-Amp\`{e}re equations}. 
     %  \newblock Proc. Amer. Math. Soc. \textbf{138} (2010), no. 10, 3637--3650.
    
    %\bibitem[RWN11]{RWN11}J. Ross and D. Witt-Nystr\"{o}m: 
     %  \newblock \emph{Analytic test configurations and geodesic rays.}
     %  \newblock J. Symplectic Geom. \textbf{12} (2014), no. 1, 125--169. 
    
    %\bibitem[R09]{Rub09}Y.~A. Rubinstein: 
    %    \newblock \emph{On the construction of Nadel multiplier ideal sheaves and the limiting behavior of the Ricci flow}. 
     %   \newblock Trans. Amer. Math. Soc. \textbf{361} (2009), 5839--5850. 
    
     \bibitem[RTZ20]{RTZ20}Y.~A. Rubinstein, G. Tian, and Zhang: 
     \newblock \emph{Basis divisors and balanced metrics}. 
     \newblock To appear in J. Reine Angew. Math.. 

    \bibitem{ST19}S. Saito and R. Takahashi: 
      \newblock  \emph{Stability of anti-canonically balanced metrics.}
       \newblock Asian J. Math. \textbf{23} (2019), no. 6, 1041--1058. 
    
    \bibitem{Sem92}S. Semmes: 
        \newblock \emph{Complex Monge--Amp\`{e}re and symplectic manifolds.}
        \newblock Amer. J. Math. \textbf{114} (1992), no. 3, 495--550.
    
    %\bibitem[ST08]{ST08}N. Sesum and G. Tian: 
    %     \newblock \emph{Bounding scalar curvature and diameter along the K\"ahler Ricci flow (after Perelman)}. 
     %    \newblock J. Inst. Math. Jussieu \textbf{7} (2008), no. 3, 575--587.
    
    %\bibitem[Sz\'e06]{Sze06}G.~Sz\'ekelyhidi: 
    %    \newblock  \emph{Extremal metrics and K-stability.}
    %    \newblock \texttt{arXiv:0611002}. Ph.D Thesis.
        
    %\bibitem[Sz\'e08]{Sze08}G.~Sz\'ekelyhidi: 
    %    \newblock  \emph{Optimal test-configurations for toric varieties}. 
    %    \newblock  J. Differential Geom. \textbf{80} (2008), no. 3, 501--523.
    
    %\bibitem[Sz\'e11]{Sze11}G.~Sz\'ekelyhidi: 
    %   \newblock  \emph{Filtrations and test-configurations.}
    %  \newblock  \texttt{arXiv:1111.4986}. 
    
    %\bibitem[ST17]{ST17}J. Song and G. Tian: 
    %\newblock \emph{The K\"ahler-Ricci flow through singularities}.
    %\newblock Invent. Math. \textbf{207} (2017), no. 2, 519--595. 
    
    %\bibitem[T97]{Tian97}G. Tian: 
    %\newblock \emph{K\"ahler-Einstein metrics with positive scalar curvature}.
    %\newblock Inv. Math. \textbf{130} (1997), 239--265. 
    
    %\bibitem[Tia15]{Tian15}G. Tian: 
    %\newblock \emph{K-stability and K\"ahler-Einstein metrics}.
    %\newblock Comm. Pure Appl. Math. \textbf{68} (2015), 1085--1156. 
    
    %\bibitem[X19]{Xia19}M. Xia: 
     %    \newblock \emph{On sharp lower bounds for Calabi type functionals and destabilizing properties of gradient flows}.  
     %    \newblock \texttt{arXiv:1901.07889}. 
        
    \bibitem{Yao17}Y. Yao: 
         \newblock \emph{Mabuchi Metrics and Relative Ding Stability of Toric Fano Varieties}. 
         \newblock International Mathematics Research Notices, 2022-24 (2022), 19790--19853. 
         
    \bibitem{Zhang21}K. Zhang: 
         \newblock \emph{A quantization proof of the uniform Yau-Tian-Donaldson conjecture}. 
         \newblock to appear in the Journal of the European Mathematical Society. 
               
    
\end{thebibliography}
\end{document}